\documentclass[11pt]{article}

 \usepackage{array}
 \usepackage{enumitem}
\newcolumntype{C}[1]{>{\centering\arraybackslash}m{#1}}
\usepackage{subfig}
\usepackage{graphicx}
\usepackage{amssymb}
\usepackage{mathtools}
\usepackage{amsmath}
\usepackage{amsthm}
\usepackage{latexsym}
\usepackage{amsfonts}
\usepackage{etoolbox}
\usepackage{multirow}
\usepackage{float}
\usepackage{amsmath,color}
\usepackage{url}
\usepackage{hyperref}
\usepackage[T1]{fontenc}

\usepackage{comment}

\usepackage{arydshln}
\usepackage{lipsum}

\newtheorem{theorem}{Theorem}{}
\newtheorem{lemma}{Lemma}{}
{}
{}
{}
{}
{}

\theoremstyle{definition}
\newtheorem{example}[theorem]{Example}
{}
\usepackage[utf8]{inputenc}
\usepackage{authblk}
\newcommand{\rank}{{\rm rank}}

\newcommand{\re}{{\rm Re}}
\newcommand{\im}{{\rm Im}}

\title{\bf A bound on the girth of quaternion unit gain graphs in terms of the rank}
\author{Suliman Khan$^{a,b,\ast}$, Edwin R. van Dam$^{b}$}
\affil{$^{a}$Department of Mathematics and Physics,\\
University of Campania ``Luigi Vanvitelli'',\\
Viale Lincoln 5, Caserta, I–81100, Italy \vspace{0.3cm}\\
$^{b}$Department of Econometrics and O.R.,\\
Tilburg University, Tilburg, Netherlands. \vspace{0.4cm}\\

 \small{Email: suliman.khan@unicampania.it,~ Edwin.vanDam@tilburguniversity.edu}}

\date{}
\begin{document}

\maketitle

\vspace{-0.1cm}

\begin{abstract}  \noindent  We obtain a bound on the girth $g$ of a quaternion unit gain graph in terms of the rank $r$ of its adjacency matrix. In particular, we show that $g \leq r+2$ and characterize all quaternion unit gain graphs for which $g=r+2$. This extends corresponding results for (ordinary) graphs, signed graphs, and complex unit gain graphs.

\end{abstract}
\vspace{0.2cm}
\thanks{{\em Math. Subj. Class.}  05C05, 05C50}\\

\section{Introduction}
\indent

We study quaternion unit gain graphs, that is, graphs with a so-called gain function that assigns a unit of the division ring of quaternions to each oriented edge. This extends the notion of signed graphs and complex unit gain graphs.

The adjacency matrix of a quaternion unit gain graph is an Hermitian matrix containing the gains of the edges in the canonical way. This extends the usual adjacency matrix of graphs and signed graphs in the same way as for complex unit gain graphs.

Studying the rank of graphs dates back to Collatz and Sinogowitz \cite{Collatz1957}, who proposed to characterize all singular graphs. Recent results that relate the rank of graphs in relation to the girth --- the size of the shortest cycle --- have been obtained by Zhou, Wong, and Tam \cite{Zhou2021} and Chang and Li \cite{Chang2022}. The rank of graphs has been studied in relation to other graph parameters as well, such as the matching number (see \cite{ Li2019, Ma2020, Song2015,Wang2014}) and the maximum degree (see \cite{Cheng2022, Sun2020, Wang2020, WangGuo2020,  Zhou2018}).

Also the rank of signed graphs and complex unit gain graphs has been extensively studied; see \cite{Chen2022, He2019, Wang2019, Wu2022} and
\cite{  He2020c, He2022c, Khan2024c, Li2022c, Lu2021c, Lu2019c, Lu2021c1,    Reff2012,  Xu2020c, Yu2015c}, respectively.

In this paper, we obtain a bound on the girth of a quaternion unit gain graph in terms of the  rank of its adjacency matrix, and characterize the extremal gain graphs. This generalizes the corresponding results for ordinary graphs by Zhou et al. \cite{Zhou2021}, signed graphs by Wu, Lu, and Tam \cite{Wu2022}, and complex unit gain graphs by Khan \cite{Khan2024c}.

\subsection{Quaternions}

The (non-commutative) division ring $\mathbb{H}$ of quaternions can be described as a 
$4$-dimensional vector space over the real numbers $\mathbb{R}$ with basis $1,i,j$, and $k$, where multiplication is defined by (extending) the following rules:

 $$i^2=j^2=k^2=ijk=-1;$$
 $$ij=k=-ji,~ jk=i=-kj,~ ki=j=-ik.$$
The quaternions thus extend the complex numbers $\mathbb{C}$.

Let $q=a_0+a_1i+a_2j+a_3k \in \mathbb{H}$. We first note that $q=(a_0+a_1i)+(a_2+a_3i)j$, which explicitly shows that $\mathbb{H}$ can be seen as a $2$-dimensional vector space over $\mathbb{C}$. The real and imaginary parts of $q$ are defined by $\re(q)=a_0$ and $\im(q)=a_1i+a_2j+a_3k$, respectively. The {\it conjugate} $q^*$ of $q$ is defined by  $q^*=a_0-a_1i-a_2j-a_3k$. We note that $(pq)^*={q}^*{p}^*$ for all $p,q \in \mathbb{H}$. The {\it modulus} $|q|$ of $q$ is given by $|q|=\sqrt{q{q}^*}=\sqrt{a_0^2+a_1^2+a_2^2+a_3^2}$. 
The  {\it circle group} of unit quaternions is denoted by $U(\mathbb{H})$, i.e., $U(\mathbb{H})=\{q \in \mathbb{H}: |q|=1\}$.

Let $A \in \mathbb{H}^{m\times n}$ be a matrix with quaternion entries. The {\it rank} of $A$ is defined as the minimal $r$ such that there is an $m \times r$ matrix $B$ and an $r \times n$ matrix $C$ such that $A=BC$. It follows that the rank equals the so-called {\it left row rank} of $A$ --- the largest number of rows of $A$ that are left linearly independent (where left refers to the coefficients multiplying from the left) --- and the similarly defined {\it right column rank}. It also follows that the rank of $A$ is the same as the rank of its conjugate transpose $A^*$.

\subsection{Quaternion unit gain graphs}

A quaternion unit gain graph is a pair $G^\varphi=(G, \varphi)$, where $G=(V(G),E(G))$ is an ordinary graph --- the {\it underlying graph} of $G^\varphi$ --- and $\varphi: \overrightarrow{E} \rightarrow U(\mathbb{H})$ is a function --- the {\it gain function} --- that assigns a unit quaternion from the circle group $U(\mathbb{H})$ to each (oriented) edge in $\overrightarrow{E}$ such that  $\varphi(e_{jk})={\varphi(e_{kj})}^*$ for each $e_{jk} \in \overrightarrow{E}(G)$. For a quaternion unit gain graph $G^\varphi$ with $V(G)=\{v_1,\dots,v_n\}$, its {\it adjacency matrix} is an $n \times n$ Hermitian matrix denoted by $A(G^\varphi)=(\ell_{jk})_{n \times n}$ and defined as follows:
$$\ell_{jk}=\left \{ \begin{array}{rcl}
\varphi(e_{jk}) & \mbox{if $v_j$ is adjacent to $v_k$}, \\
0 & \mbox{otherwise.}\\
\end{array} \right.$$

For background on matrices of quaternions, we refer to the paper by Zhang \cite{Zhang97}.
Belardo et al.~\cite{Belardo2022} recently considered the spectra of quaternion unit gain graphs, among others of the adjacency matrix. Kyrchei et al.~\cite{Kyrchei2024} studied the determinant of the Laplacian matrix of quaternion unit gain graphs. Zhou and Lu \cite{Zhou2023} studied the rank of quaternion unit gain graphs in relation to the rank of their underlying graphs. In our study, we consider the rank of quaternion unit gain graphs in relation to the girth.

\subsection{The main result}

As usual, for a walk $W$ in $G^\varphi$, we let its gain $\varphi(W)$ be the product of the gains of the (oriented) edges in the walk, where we multiply the gain in each step of the walk by the corresponding edge gain from the right (this corresponds naturally with multiplication from the right by the adjacency matrix). We note that the gain of a cycle depends (in general) on its starting point and direction, but not its real part (see \cite{Belardo2022}). Thus, $\re(\varphi(C))$ is well defined for a cycle $C$ in $G^\varphi$. Moreover, also the property that $\varphi(C) \in \mathbb{R}$ is independent of direction and starting point.

\begin{theorem}\label{Theorem1}
Let $G^{\varphi}$ be a connected quaternion unit gain graph with girth\footnote{We assume that the underlying graph is not a tree} $g$ and rank $r$. Then $g \leq r+2$, with equality if and only if $g$ is even and one of the following holds.\\
(i) $G^{\varphi}$ is a $g$-cycle with gain $(-1)^{\frac{g}{2}}$.\\
(ii) $G^{\varphi}$ is a complete bipartite graph with both biparts having order at least $2$, in which all $4$-cycles have gain $1$.
\end{theorem}

We note that this result also applies to complex unit gain graphs (note that in \cite{Khan2024c}, there are some incorrect definitions of types of gain graphs).

The subsequent sections are arranged as follows.
 In Section $2$, we recall the method of switching in gain graphs and consider the ranks of paths and cycles. In Section 3, we present the proof of Theorem \ref{Theorem1}.

\section{Switching, paths, and cycles}

A switching function on a (quaternion unit) gain graph with vertex set $V$ is a function $\vartheta: V \rightarrow U(\mathbb{H})$. Switching a quaternion unit gain graph $G^\varphi$ by the switching function $\vartheta$ results in a new quaternion unit gain graph ${G^\varphi}^\vartheta$, in which the underlying graph remains unchanged but the gain of any edge $xy$ is changed into $\varphi^\vartheta(xy)={\vartheta(x)}^{*}\varphi(xy)\vartheta(y)$. 
The adjacency matrix of ${G^\varphi}^\vartheta$ can thus be expressed as $A({G^\varphi}^\vartheta) =D^*AD$, where $D$ 
is the diagonal matrix containing the switching function values $\vartheta(v), v \in V,$ and $A$ is the adjacency matrix of $G^\varphi$. Hence, switching does not change the rank.


It is well-known (for other types of gain graphs) and easy to see that any gain tree is switching equivalent to its underlying graph, and hence the rank of a gain tree is the same as the rank of the underlying graph. In particular, the ranks of gain paths follow.

\begin{lemma}\label{LemmaP1}
Let $P_n^{\varphi}$ be a quaternion unit gain path of order $n$ and rank $r$. Then $r=n-1$ if $n$ is odd, and $r=n$ if $n$ is even.
\end{lemma}

\begin{proof}
    By switching, this follows from the corresponding result for ordinary paths, which follows for example from their eigenvalues \cite{bh12}.
\end{proof}

Zhou and Lu \cite{Zhou2023} determined the rank of quaternion unit gain cycles. Their relatively technical proof can be simplified using switching, for any cycle can be switched into a cycle in which all but one of the edges have gain $1$, and where the remaining edge gain contains all the gain of the original cycle. Indeed, if the original cycle $C$ has vertices $v_0,v_1,v_2, \dots ,v_n=v_0$ such that $v_iv_{i+1}$ is an edge with gain $\varphi(v_iv_{i+1})=:\phi_i$ for $i=0,\dots,n-1$, then by defining the switching function $\vartheta$ recursively by $\vartheta(v_0)=1$ and $\vartheta(v_{i+1})=\phi_i^*\vartheta(v_{i})$ for $i=0,\dots,n-2$, we obtain that after switching only the final edge $v_{n-1}v_n$ has (possibly) non-one gain, and this gain equals $h=\phi_0\phi_1 \cdots \phi_{n-1}$, that is, the gain of the closed walk on the original cycle $C$ starting at $v_0$. As noted before, this gain does depend on the direction and starting point; however $\re(\varphi(C))$ is well defined and also the property that $\varphi(C) \in \mathbb{R}$ is independent of direction and starting point.

\begin{lemma}\label{Lemma3}
\cite{Zhou2023} Let $C$ be a  quaternion unit gain $n$-cycle with rank $r$. Then\\
\begin{equation*}
 r=
  \begin{cases}
    n-2, & \text{ if $n$ is even and} \ ~ \varphi(C)=(-1)^{\frac{n}{2}};\\
    n, & \text{ if $n$ is even and} \ ~ \varphi(C) \neq (-1)^{\frac{n}{2}};\\
    n, & \text{ if $n$ is odd and} \ ~ \re(\varphi(C)) \neq 0;\\
    n-1, & \text{ if $n$ is odd and} \ ~ \re(\varphi(C)) = 0.
  \end{cases}\vspace{0.38cm}
\end{equation*}
\end{lemma}

\begin{proof}
    (Sketch). We first switch $C$ into a cycle $C_n(h)$, with all but one edges having gain $1$ and the remaining one having gain $h=\varphi(C)$ (as explained above). By row reduction it follows that  $\rank(C_n(h))=2+\rank(C_{n-2}(-h)$. By induction and considering the ranks of $C_3(\pm h)$ and $C_4(\pm h)$, the result follows.
\end{proof}

Finally, we need the following straightforward result.

\begin{lemma}\label{Lemmapendantcycle}
Let $G^{\varphi}$ be a quaternion unit gain graph whose underlying graph is an $n$-cycle with a pendant edge. Then $\rank(G^{\varphi})=2+\rank(P_{n-1}) \geq n$.
\end{lemma}

\begin{proof}
    This follows from row reduction using the rows corresponding to the vertices of the pendant edge, and applying Lemma \ref{LemmaP1}.
\end{proof}

\section{Proof of the main result}

\subsection{The bound}

Let $G^\varphi$ be a quaternion unit gain graph with girth $g$ and rank $r$. The bound $g \leq r+2$ in our main result follows easily by observing that the rank of any induced subgraph (whose adjacency matrix is a submatrix of $A(G^\varphi)$), in particular a cycle of order $g$, is at most the rank of $G^\varphi$ and then applying Lemma \ref{Lemma3}. By the same lemma, we find that equality can only occur if $g$ is even and every cycle of order $g$ has gain $(-1)^{g/2}$. What remains is to find all gains graphs where this occurs.

\subsection{Equality}

It will turn out that equality can only occur if the underlying graph is an even cycle or a complete bipartite graph. We start by characterizing the gain graphs on the latter family.

\begin{lemma}\label{Lemmabipartite}
 Let $G^\varphi$ be a connected bipartite quaternion unit gain graph with rank $r$. Then $r=2$ if and only if $G^\varphi$ is complete bipartite and all $4$-cycles have gain $1$.
\end{lemma}

\begin{proof}
Because the underlying graph is bipartite, there is a matrix $B$ such that 
 $$A(G^\varphi)=\begin{bmatrix}
                     O & B \\
                     B^* & O
                   \end{bmatrix}.$$
As $\rank(A)=2~\rank(B)$, it is clear that we require the rank of $B$ to be $1$, in other words, the rows are multiples of each other. First of all, this implies that $B$ cannot have any zeros, for otherwise there would be isolated vertices, which is a contradiction. Thus, $G^\varphi$ must be complete bipartite.

Next, note that the statement is true if one of the biparts of the bipartite graph has only one vertex. We may thus assume that both biparts of the bipartite graph have more than one vertex.

Consider two arbitrary rows of $B$; say they correspond to vertices $x_1,x_2$ in one bipart. Consider two arbitrary vertices in the other bipart, $y_1,y_2$, say.
Let $c_1=\varphi(x_1y_1)\varphi(x_2y_1)^*$ and $c_2=\varphi(x_1y_2)\varphi(x_2y_2)^*$. Then the cycle through these four vertices has gain $\varphi(x_1y_1x_2y_2x_1)=c_1c_2^*$, which equals $1$ precisely when $c_1=c_2$. It follows that the two rows are multiples of each other precisely when all $4$-cycles through $x_1$ and $x_2$ have gain $1$, which proves the entire statement.
\end{proof}

We are now ready to characterize all quaternion unit gain graphs with $g=r+2$.

\begin{proof}[{Proof of Theorem \ref{Theorem1}}]
We recall that the bound $g \leq r+2$ follows because $G^\varphi$ has a $g$-cycle as induced subgraph, and that equality can only occur if $g$ is even and all $g$-cycles have gain $(-1)^{g/2}$. One such graph is a $g$-cycle with gain $(-1)^{g/2}$ and by Lemma \ref{Lemmabipartite}, equality occurs for complete bipartite graphs in which all $4$-cycles have gain $1$. What therefore remains is to prove the reverse. 

Assume that $G^\varphi$ is not a $g$-cycle. Then $G^\varphi$ contains an induced $g$-cycle and because of connectivity, there is a vertex that is adjacent to at least one of the vertices of this cycle. This vertex must be adjacent to at least two of the vertices of the cycle, for otherwise $G^\varphi$ would have an induced subgraph of rank at least $g$ by Lemma \ref{Lemmapendantcycle}, which is a contradiction. As this vertex is now part of a cycle with other vertices of the $g$-cycle (a shortest cycle), it follows that $g \leq 4$. Because $g$ is even, we therefore conclude that $g=4$ and $r=2$.

Next, suppose that $G^\varphi$ contains odd cycles. Consider a shortest odd cycle, say of size $g_o$. Because this cycle is a shortest odd cycle, it must be an induced cycle. Moreover, $g_o \geq 5$, and hence by Lemma \ref{Lemma3}, $G^\varphi$ has an induced subgraph of rank at least $4$, which is a contradiction. Thus, $G^\varphi$ must be bipartite, and then it follows from Lemma \ref{Lemmabipartite} that $G^\varphi$ must be complete bipartite with all $4$-cycles having gain $1$.
\end{proof}


\end{document}